\documentclass[12pt,a4paper]{article}
\usepackage[latin5]{inputenc}
\usepackage[T1]{fontenc}
\usepackage{amsfonts,amssymb,amsthm}
\usepackage{amsmath}
\numberwithin{equation}{section}
\usepackage{graphicx}
\usepackage{enumerate,enumitem}
\usepackage{multicol}
\usepackage[all]{xy}
\usepackage{xcolor,soul,faktor,xfrac}
\usepackage{bm}
\usepackage[bookmarks=true,
bookmarksnumbered=true, breaklinks=true,
pdfstartview=FitH, hyperfigures=false,
plainpages=false, naturalnames=true,
colorlinks=true,citecolor=blue,%pagebackref=true,
pdfpagelabels]{hyperref}
\usepackage[sort,numbers,compress]{natbib}
\usepackage{geometry}
\allowdisplaybreaks
\usepackage{indentfirst,tocloft}
\usepackage{authblk}
\usepackage{lmodern}
\usepackage{bbold}
\usepackage{shortlst}
\usepackage{easyReview}
\usepackage{multicol}
\usepackage{setspace}
\theoremstyle{definition}
\newtheorem{theorem}{Theorem}[section]
\newtheorem{proposition}[theorem]{Proposition}
\newtheorem{lemma}[theorem]{Lemma}
\newtheorem{corollary}[theorem]{Corollary}

\newtheorem{definition}[theorem]{Definition}
\newtheorem{example}[theorem]{Example}
\newtheorem{remark}[theorem]{Remark}

\def\Gwa{\mathbf{Gr}^{\bullet}}
\def\rGwa{\mathbf{rGr}^{\bullet}}
\def\Gr{\mathbf{Gr}}

\def\PA{\mathsf{Pentact}}
\def\USGA{\mathsf{USGA}}
\def\St{\mathsf{Stab}}
\def\wSt{\mathsf{wStab}}
\def\Ann{\mathsf{Ann}}
\def\AR{\mathbb{AR}}
\def\bb{\mathbb{b}}
\def\bbz{\mathbb{0}}
\def\Act{\operatorname{Actor}}
\def\oa{\overline{a}}

\title{\textbf{Pentactions and action representability in the category of reduced groups with action}}

\author[a]{Tamar Datuashvili\thanks{\textbf{Corresponding author: }\texttt{tamar.datu@gmail.com} (T. Datuashvili)}}
\author[b]{Tunçar Şahan}
\affil[a]{\small{A. Razmadze Mathematical Institute of I. Javakhishvili Tbilisi State University, 6 Tamarashvii Str., Tbilisi 0177, Georgia}}
\affil[b]{\small{Department of Mathematics, Aksaray University, Aksaray, Turkey}}
\date{}

\begin{document}
\maketitle

\begin{abstract}
A notion of pentaction of any object in the category $\rGwa$ of reduced groups with action is introduced. The operations are defined in the set $\PA(A)$ of pentactions of an object $A$ of $\rGwa$. It is proved that if an object $A$ is perfect with zero weak stabilizer in the sense defined in the paper, then $\PA(A)$ is an object of $\rGwa$, it has a derived action on $A$, the object $A$ is action representable and $\PA(A)$ represents all actions on $A$.\\[0.4cm]
\end{abstract}

\section{Introduction}

Reduced groups with action were introduced in \cite{Datuashvili2021} as a certain full subcategory $\rGwa$ of the category of groups with action, which was defined in \cite{Datuashvili2002a,Datuashvili2004} and played a crucial role in the solution of Loday's two problems stated in \cite{Loday1993,Loday2003}. The aim of the definition of $\rGwa$ was investigation of representability of actions in this category. The notion of representable action was introduced in \cite{Borceux2005} in terms of monoidal categories and applied to semi-abelian categories. In \cite{Casas2007,Casas2010} was developed action theory in categories of interest in the sense of \cite{Orzech1972,Orzech1972a}, where was given a definition and construction of universal strict general actor $\USGA(A)$ of any object $A$ in the category of interest and found its properties. This construction and the corresponding results were applied to various cases of algebraic categories in order to prove the existence of universal acting objects, called actor, under appropriate conditions and to give their constructions \cite{Atik2017,Boyaci2015,Casas2007,Casas2010,Casas2009,Casas2009a,Casas2013,Casas2012,Casas2019}. The notion of actor is equivalent to split extension classifier in semi-abelian categories \cite{Borceux2005}. The category $\Gwa$ is not a category of interest, therefore we found its subcategory $\rGwa$, which is also not a category of interest, but it has some properties, which are similar to those which are important in the construction of $\USGA(A)$ of an object $A$. In this paper we apply the results obtained in \cite{Casas2010} and \cite{Datuashvili2021}, and introduce a notion of pentactions of $A$ in $\rGwa$. It is proved that if an object $A$ is perfect with zero weak stabilizer in the sense defined in the paper, then pentactions of $A$ form an object of $\rGwa$, denoted by $\PA(A)$, $A$ is action representable and $\PA(A)$ represents all actions on $A$.

In Section \ref{sec:prelim} we recall the definitions of the categories $\Gwa$ and $\rGwa$, derived action in $\rGwa$ and describe its properties. We give a definition of representable action in $\rGwa$ and of action representer, which agrees with the definition given in \cite{Borceux2005} in more general setting.

In Section \ref{sec:pentact} we define a notion of pentaction of any object $A\in\rGwa$. The notion of perfect objects of $\rGwa$ is introduced and an example is given. This notion is an analogue of perfect algebras in the categories of associative and Leibniz algebras, respectively. This property is a sufficient condition for the existence of an actor for the objects in these categories \cite{Casas2010,Datuashvili2017}. Then we introduce notions of stabilizer and weak stabilizer of an object $A\in\rGwa$ denoted by $\St(A)$ and  $\wSt(A)$, respectively. The prototype of stabilizer of $A$ is annihilator in the case of algebras; recall that $\Ann(A)=0$ is a sufficient condition for the existence of actor in the categories of associative and Leibniz algebras (see e.g. \cite{Casas2010,Datuashvili2017}). We give an example of an object of $\rGwa$ which is perfect with zero weak stabilizer. It is proved that if $A$ is perfect with $\wSt(A)=0$, then $\PA(A)$ is an object of $\rGwa$ (Theorem \ref{theo:pentactrgwa}).

In Section \ref{sec:actrep} we define an action of $\PA(A)$ on $A$ and prove that if $A$ is perfect with $\wSt(A)=0$, then this action is a derived action in $\rGwa$ (Theorem \ref{theo:pentactderact}). At the end of this section it is proved that under the same conditions on the object $A$, $A$ has representable actions and $\PA(A)$ represents all actions on $A$ in $\rGwa$ (Theorem \ref{theo:pentactrepresent}).

\section{Preliminary definitions and results}\label{sec:prelim}

The category $\Gwa$ of groups with action is defined in \cite{Datuashvili2002a} as a category, where objects are groups which act on itself. Morphisms in this category are group homomorphisms, which preserve an action operation, i.e. $\alpha\colon G\rightarrow G'$ is a group homomorphism with the property $\alpha({g}^{h})={\alpha(g)}^{\alpha(h)}$, where $G,G'\in\Gwa$ and $g,h\in G$. The group operation is denoted additively, but it is not commutative in general. Recall that an action $G'\times G\rightarrow G$ in the category $\Gr$ of groups satisfies the following conditions 
\begin{alignat*}{2}
	{(g+g')}^{h} & = {g}^{h}+{g'}^{h} \\
	{g}^{h+h'} & = {\left({g}^{h}\right)}^{h'} \\
	{g}^{0} & = g
\end{alignat*}
for any $G,G'\in\Gr$ and any $g,g'\in G$, $h,h'\in G'$, where 0 is the zero element of $G$. Note that from the first condition it follows that ${0}^{h}=0$, for any $h\in G'$.

A category $\rGwa$ of reduced groups with action is introduced in \cite{Datuashvili2021} as a full subcategory of $\Gwa$ of those objects $G$, which has the following two properties
\begin{alignat*}{2}
	{x}^{y}+z & = z+{x}^{y} \,\, (y\neq 0) \\
	{x}^{\left({y}^{z}\right)} & = {x}^{y} 
\end{alignat*}
for any $x,y,z\in G$.

Note that a group with action on itself by conjugation or with trivial action is not an object of $\rGwa$. Any abelian group with trivial action is an object of $\rGwa$.

The category $\Gwa$ is not a category of interest. In \cite{Datuashvili2021} we modified for this category the definition of derived action due to split extensions known for the category of interest \cite{Orzech1972a} or category of groups with operations \cite{Porter1987}. In analogous way derived action is defined in $\rGwa$. This definition  derived from split extension agrees with the definition of action in a semi-abelian category \cite{Borceux2005}.

Let $A,B\in\rGwa$. An extension of $B$ by $A$ is a sequence 
\begin{equation}\label{ses}
	\xymatrix{0 \ar[r] &   A \ar@{->}[r]^-{i} &   E \ar@{->}[r]^-{p} &   B \ar[r] & 0}
\end{equation}
in which $p$ is surjective and $i$ is the kernel of $p$. It is said that an extension is split if there is a morphism $j\colon B\rightarrow E$, with $pj=1_B$. Below $i(a)$ will be identified with $a$. A split extension induces a triple of action of $B$ on $A$ corresponding to the operations of addition, action and its dual operation in $\rGwa$. From the split extension \eqref{ses} for any $b\in B$ and $a\in A$, actions are defined by
\begin{equation}\label{deract}
	\begin{alignedat}{2}
		b\cdot a &= j(b)+a-j(b) \\
		{b}^{a} &={j(b)}^{a}-j(b) \\
		{a}^{b} &={a}^{j(b)}. 
	\end{alignedat}
\end{equation}
Actions defined by \eqref{deract} are called derived actions of $B$ on $A$ as it is in the case of groups with operations or category of interest. The second action in \eqref{deract} differs from what we have in the noted known cases, since ${j(b)}^{0}\neq 0$ in $E$.

Suppose we have a triple of functions $B\times A\rightarrow A$, which one can call action of $B$ on $A$ denoted by $b\cdot a$, ${a}^{b}$ and ${b}^{a}$ for any $a\in A$ and $b\in B$. According to Theorem 4.2 in \cite{Datuashvili2021}, this triple is a derived action in $\rGwa$ if and only if it satisfies group action conditions for the addition operation and also the conditions:
\begin{enumerate}[leftmargin=1.5cm]
	\item[$\boldsymbol{(1_{A})}$]\label{prop:act1a} ${(a+a')}^{b}={a}^{b}+{(a')}^{b}$,
	\item[$\boldsymbol{(2_{A})}$]\label{prop:act2a} ${\left(b+b'\right)}^{a}={b}^{a}+b\cdot\left((b')^{a}\right)$,
	\item[$\boldsymbol{(3_{A})}$]\label{prop:act3a} ${(b\cdot a)}^{a'}= {a}^{a'}$ for $a'\neq 0$,
	\item[$\boldsymbol{(4_{A})}$]\label{prop:act4a} $\left(b\cdot a\right)^{b'}={a}^{b'}$,
	\item[$\boldsymbol{(1_{B})}$]\label{prop:act1b} ${b}^{(a+a')}=\left({b}^{a}\right)^{a'}+{b}^{a'}$,
	\item[$\boldsymbol{(2_{B})}$]\label{prop:act2b} ${a}^{b+b'}=\left({a}^{b}\right)^{b'}$,
	\item[$\boldsymbol{(3_{B})}$]\label{prop:act3b} $\left({a}^{(b\cdot a')}\right)^{b}=\left({a}^{b}\right)^{a'}$,
	\item[$\boldsymbol{(4_{B})}$]\label{prop:act4b} $\left({b}^{(b'\cdot a)}\right)^{b'}=\left({b}^{b'}\right)^{a}$,
\end{enumerate}
the condition ${a}^{0_{B}}=a$ for any $a\in A$, where $0_B$ is the zero element of $B$, and the conditions 
\renewcommand{\arraystretch}{1.43}
\begin{equation*}\label{tennew}
	\begin{array}{rrclcrrcl}
		a_1. & b\cdot{a}^{a'}              & = & {a}^{a'} \text{ for } a'\neq 0  & \qquad\qquad\qquad & a_6.    & {a}^{b}+a'                  & = & a'+ {a}^{b} \text{ for } b\neq 0 \\
		a_2. & b\cdot{a}^{b'}              & = & {a}^{b'}  \text{ for } b'\neq 0 & \qquad\qquad\qquad & a_7.    & {a}^{\left({a'}^{b}\right)} & = & {a}^{a'}                         \\
		a_3. & {b}^{b'}\cdot a             & = & a  \text{ for } b'\neq 0        & \qquad\qquad\qquad & a_8.    & {a}^{\left({b}^{a'}\right)} & = & a  \text{ for } a'\neq 0         \\
		a_4. & {b}^{\left({a}^{a'}\right)} & = & {b}^{a}                         & \qquad\qquad\qquad & a_9.    & {b}^{\left({b'}^{a}\right)} & = & 0                                \\
		a_5. & {a}^{\left({b}^{b'}\right)} & = & {a}^{b}                         & \qquad\qquad\qquad & a_{10}. & {b}^{\left({a}^{b'}\right)} & = & {b}^{a}
	\end{array}	
\end{equation*}
for any $a,a'\in A$, $b,b'\in B$. Note that $\boldsymbol{(3_A)}$ and $\boldsymbol{(4_A)}$ have simpler forms in the category $\rGwa$, than they have it in the category $\Gwa$ \cite[Proposition 3.1]{Datuashvili2021} and it is noted in \cite[Theorem 4.2 (2)]{Datuashvili2021} as well. Here we applied the conditions $a_1$ and $a_{3}$ and the property of any object of $\rGwa$ ${x}^{y}+z  = z+{x}^{y}$, $x,y,z\in G$, $G\in\rGwa$.

As we have noted in the introduction representable action was introduced in monoidal categories in \cite{Borceux2005}, we give this definition for the special case in the category $\rGwa$.

\begin{definition}
	We will say that an object $A$ of $\rGwa$ has representable action in this category, if there is an object $R$ with a derived action on $A$ with the property that for any object $B$ of $\rGwa$ with a derived action on $A$ there exists a unique homomorphism in $\rGwa$
	\begin{equation*}
		\varphi\colon B\rightarrow R
	\end{equation*}
	with the property that
	\begin{equation}\label{repact}
		\begin{alignedat}{2}
			b\cdot a &= \varphi(b)\cdot a \\
			{b}^{a} &= {\varphi(b)}^{a} \\
			{a}^{b} &={a}^{\varphi(b)}
		\end{alignedat}
	\end{equation}
	for any $b\in B$ and $a\in A$. On the right side of \eqref{repact} we mean the action of $R$ on $A$.
\end{definition}

We will call the object $R$ action representer of $A$ and denote by $\AR(A)$. It is obvious that if such an object exists for a given $A$, then it is unique object up to isomorphism with this property. The notion of $\AR(A)$ is the same as the notion of actor of the object $A$, denoted by $\Act(A)$, in the case of category of interest \cite{Casas2007,Casas2010} and it is equivalent to the special case of the notion of split extension classifier for the object defined in semi-abelian categories in pure categorical way \cite{Borceux2005}.

\section{Pentactions in $\rGwa$}\label{sec:pentact}

Let $A\in\rGwa$. Consider the functions
\begin{equation*}
	b\cdot \text{-} ~,~ \text{-}\cdot b ~,~ {(\text{-})}^{b} ~,~ {\vphantom{(-)}}^{b}{(\text{-})} ~,~ {b}^{(\text{-})} \colon A\rightarrow A.
\end{equation*}

The values of these functions will be denoted by $b\cdot a, a\cdot b, {a}^{b}, {}^{b}{a}, {b}^{a}$ for any $a\in A$.

\begin{definition}\label{def:pentcond}
	We will say that $\bb=\left(b\cdot \text{-},\text{-}\cdot b,{(\text{-})}^{b},{\vphantom{
			(-)}}^{b}{(\text{-})},{b}^{(\text{-})}\right)$ is a pentaction of $A$ if the following conditions are satisfied, where $a,a'\in A$.
	
	\setstretch{1.5}
	\begin{shortenumerate}[aaaaaaaaaaaaaaaaaaaaaaaaaaaaaaaaaaaaaaaaa]
		\item[1.] $b\cdot(a+a')=b\cdot a+b\cdot a'$
		\item[1$^{\circ}$.] $(a+a')\cdot b = a\cdot b+a'\cdot b$
		\item[2.] ${(a+a')}^{b}={a}^{b}+{a'}^{b}$ \label{cond:2pentact}
		\item[2$^{\circ}$.] ${}^{b}{(a+a')}={}^{b}{a}+{}^{b}{a'}$ \label{cond:2cpentact}
		\item[3.] ${(b\cdot a)}^{a'}={a}^{a'}$ for any $a'\neq 0$
		\item[3$^{\circ}$.] ${(a\cdot b)}^{a'}={a}^{a'}$ for any $a'\neq 0$
		\item[4.] ${b}^{(a+a')}={\left({b}^{a}\right)}^{a'}+{b}^{a'}$
		\item[]
		\item[5.] ${\left({a}^{b\cdot a'}\right)}^{b}={\left({a}^{b}\right)}^{a'}$
		\item[5$^{\circ}$.] ${\vphantom{\left({}^{a'\cdot b}{a}\right)}}^{b}{\left({}^{a'\cdot b}{a}\right)}={\vphantom{\left({}^{a'\cdot b}{a}\right)}}^{a'}{\left({}^{b}{a}\right)}$
		\item[6.] $b\cdot {a}^{a'}={a}^{a'}$ for any $a'\neq 0$
		\item[6$^{\circ}$.] ${a}^{a'}\cdot b={a}^{a'}$ for any $a'\neq 0$
		\item[7.] ${b}^{\left({a}^{a'}\right)}={b}^{a}$
		\item[]
		\item[8.] \parbox[t]{0.8\shortitemwidth}{${a}^{b}+a'=a'+{a}^{b}$, if ${a}^{b}\neq a$ at least for\\[-2mm] one $a\in A$\\[-5mm]}
		\item[8$^{\circ}$.] \parbox[t]{0.8\shortitemwidth}{${}^{b}{a}+a'=a'+{}^{b}{a}$, if ${}^{b}{a}\neq a$ at least for\\[-2mm] one $a\in A$\\[-5mm]}
		\item[9.] ${a}^{\left({a'}^{b}\right)}={a}^{a'}$
		\item[9$^{\circ}$.] ${a}^{\left({}^{b}{a'}\right)}={a}^{a'}$
		\item[10.] ${a}^{\left({b}^{a^{\prime}}\right)}=a$ for any $a'\neq 0$
		\item[]
		\item[11.] $(b\cdot a)\cdot b=b\cdot (a\cdot b)=a$
		\item[]
		\item[12.] ${\vphantom{\left({}^{b}{a}\right)}}^{b}{\left({a}^{b}\right)}={\left({}^{b}{a}\right)}^{b}=a$
		\item[]
	\end{shortenumerate}
	\setstretch{1.1}
\end{definition}

Denote the set of all pentactions of $A$, $A\in\rGwa$, by $\PA(A)$.

Note that it follows from conditions $2.$ and $2^{\circ}.$ that ${0}^{b}={}^{b}{0}=0$, and it follows from $4$ that ${b}^{0}=0$, for any ${(\text{-})}^{b},{\vphantom{(\text{-})}}^{b}{(\text{-})}$ and ${b}^{(\text{-})}\in\bb$, $\bb\in\PA(A)$.

Consider the set of functions
\begin{equation*}
	\bbz=\left(0\cdot \text{- },\text{-}\cdot 0,{(\text{-})}^{0},{\vphantom{(\text{-})}}^{0}{(\text{-})},{0}^{(\text{-})}\right),
\end{equation*}
where $0\cdot a = a$, $a\cdot 0=a$, ${a}^{0}={}^{0}{a}=a$, ${0}^{a}=0$. It is easy to see that $\bbz$ is a pentaction; conditions $8$ and $8^{\circ}$ hold for $\bbz$, since by definition ${a}^{0}={}^{0}{a}=a$ for any $a\in A$.

For any $\bb,\bb'\in\PA(A)$ define a new set of functions, denoted by 

\begin{equation*}
	\bb+\bb'=\left(\left(b+b'\right)\cdot \text{-}~,\text{-}\cdot \left(b+b'\right),{(\text{-})}^{\left(b+b'\right)},{\vphantom{(\text{-})}}^{\left(b+b'\right)}{(\text{-})},{\left(b+b'\right)}^{(\text{-})}\right),
\end{equation*}
where by definition for any $a\in A$
\begin{equation}\label{eq:bpb}
	\begin{alignedat}{2}
		\left(b+b'\right)\cdot a &= b\cdot (b'\cdot a) \\
		a\cdot \left(b+b'\right) &= (a\cdot b)\cdot b'\\ 
		{a}^{\left(b+b'\right)} &= {\left({a}^{b}\right)}^{b'}\\
		{\vphantom{a}}^{\left(b+b'\right)}{a} &= {\vphantom{\left({\vphantom{a}}^{b'}{a}\right)}}^{b}{\left({\vphantom{a}}^{b'}{a}\right)}\\
		{\left(b+b'\right)}^{a} &= {b}^{a} + b\cdot\left({b'}^{a}\right).
	\end{alignedat}
\end{equation}

For any $\bb,\bb'\in\PA(A)$ consider a new set of functions 
\begin{equation*}
	{\bb}^{\bb'}=\left(\left({b}^{b'}\right)\cdot \text{- },\text{-}\cdot \left({b}^{b'}\right),{(\text{-})}^{\left({b}^{b'}\right)},{\vphantom{(\text{-})}}^{\left({b}^{b'}\right)}{(\text{-})},{\left({b}^{b'}\right)}^{(\text{-})}\right),
\end{equation*}
defined by 
\begin{equation}\label{eq:btob}
	\begin{alignedat}{2}
		\left({b}^{b'}\right)\cdot a &= a \\
		a\cdot \left({b}^{b'}\right) &= a\\ 
		{a}^{\left({b}^{b'}\right)} &= {a}^{b}\\
		{\vphantom{a}}^{\left({b}^{b'}\right)}{a} &= {\vphantom{a}}^{b}{a}\\
		{\left({b}^{b'}\right)}^{a} &= {\left({b}^{\left(b'\cdot a\right)}\right)}^{b'},
	\end{alignedat}
\end{equation}
for any $a\in A$.

\begin{definition}
	An object $A$ of the category $\rGwa$ will be called perfect if ${A}^{A}=A$.
\end{definition}

Here under ${A}^{A}$ is meant the subobject of $A$ in $\rGwa$ generated by the elements of the type ${a}^{a'}$, $a,a'\in A$.

\begin{example}\label{ex:abeltriv}
	Any abelian group with trivial action is perfect in $\rGwa$.
\end{example}

\begin{example}
	For any $A\in\rGwa$, the object ${A}^{A}$ is a perfect object. We have to show only that ${A}^{A}\subseteq {\left({A}^{A}\right)}^{\left({A}^{A}\right)}$. Any element of ${A}^{A}$ of the type ${{a}_{1}}^{{a}_{2}}$, ${a}_{1},{a}_{2}\in A$, can be represented as ${\left({{a}_{1}}^{{a}_{2}-y}\right)}^{\left({y}^{z}\right)}$ for any $y,z\in A$, which means that ${{a}_{1}}^{{a}_{2}}\in{\left({A}^{A}\right)}^{\left({A}^{A}\right)}$. The general case follows directly from this result, which proves that ${A}^{A}={\left({A}^{A}\right)}^{\left({A}^{A}\right)}$.
\end{example}

This notion is an analogue of the well-known notions of perfect objects in the categories of associative or Leibniz algebras, which are algebras with representable actions, i.e. which have actors in the corresponding categories (see e.g. \cite{Casas2010,Datuashvili2017}).

\begin{lemma}\label{lem:perfect}
	If $A$ is a perfect object of $\rGwa$, then $b\cdot a=a$ for any $b\cdot\text{-}\in\bb$, $\bb\in\PA(A)$.
\end{lemma}
\begin{proof}
	The result follows from the condition $6$ of Definition \ref{def:pentcond}.
\end{proof}

\begin{corollary}\label{cor:changeofpower}
	${\left({b}^{b'}\right)}^{a}={\left({b}^{a}\right)}^{b'}$, for any $a\in A$, ${\left({b}^{b'}\right)}^{(\text{-})}\in{\bb}^{\bb'}$, $\bb,\bb'\in\PA(A)$, ${b}^{(\text{-})}\in\bb$, ${(\text{-})}^{b'}\in\bb'$.
\end{corollary}
\begin{proof}
	Here we apply the definition of ${\bb}^{\bb'}$, where 
	\begin{equation*}
		{\left({b}^{b'}\right)}^{a}={\left({b}^{b'\cdot a}\right)}^{b'}
	\end{equation*}
	and by Lemma \ref{lem:perfect} we have $b'\cdot a =a$.
\end{proof}

\begin{corollary}
	${\left({a}^{a'}\right)}^{b}={\left({a}^{b}\right)}^{a'}$, where ${(\text{-})}^{b}\in\bb$, $\bb\in\PA(A)$.
\end{corollary}
\begin{proof}
	It follows from the condition $5$ of Definition \ref{def:pentcond} and Lemma \ref{lem:perfect}.
\end{proof}

\begin{definition}
	A stabilizer of the object $A\in\rGwa$, denoted by $\St(A)$, is a set of elements $\oa\in A$ with the property that ${a}^{\oa}=a$, for any $a\in A$, i.e.
	\begin{equation*}
		\St(A)=\left\{\oa\in A \;\middle|\; {a}^{\oa}=a, \text{ for any } a\in A \right\}.
	\end{equation*}
\end{definition}

Note that a notion of stabilizer of an element of a set is known for the case group with action on a set. We will see that annihilation of certain type stabilizers, defined below, will play analogous role in the category of reduced groups with action, as annihilators play in the case of associative and Leibniz algebras for the existence of actors in these categories; here we mean the well-known condition $\Ann(A)=0$ in the corresponding categories (see e.g. \cite{Casas2010,Datuashvili2017}).

\begin{lemma}
	Let $A\in\rGwa$. If $\St(A)=0$, then $A=0$.
\end{lemma}
\begin{proof}
	For any $a,a_1$ and $a_2$ from $A$ we have 
	\begin{equation*}
		{a}^{\left({a_{1}}^{a_2}-a_1\right)}={\left({a}^{a_1}\right)}^{-a_1}=a;
	\end{equation*}
	therefore ${a_1}^{a_2}-a_1=0$, since $\St(A)=0$. From this it follows that ${a_1}^{a_2}=a_1$, which means that the action of $A$ on itself is trivial; it implies that $A=\St(A)=0$.
\end{proof}

\begin{definition}
	A weak stabilizer of an object $A\in\rGwa$ is defined as follows;
	\begin{equation*}
		\wSt(A)=\left\{{b}^{\left({b'}^{a}\right)}, {b}^{\left({a}^{b'}\right)}-{b}^{a}, {a}^{b+b'}-{a}^{b'+b}\right\},
	\end{equation*}
	where $a\in A$, ${b}^{(\text{-})}\in\bb$, ${b'}^{(\text{-})},{(\text{-})}^{b'}\in\bb'$, ${(\text{-})}^{b+b'}\in\bb+\bb'$, ${(\text{-})}^{b'+b}\in\bb'+\bb$, $\bb,\bb'\in\PA(A)$. 
\end{definition}

\begin{lemma}
	$\wSt(A)$ is a subset of $\St(A)$, for any $A\in\rGwa$.
\end{lemma}
\begin{proof}
	For any $\oa\in A$, we have ${\oa}^{\left({b}^{\left({b'}^{a}\right)}\right)}=\oa$ by Definition \ref{def:pentcond}, condition 10. Applying the same condition we have
	\begin{equation*}
		{\oa}^{\left({b}^{\left({a}^{b'}\right)}-{b}^{a}\right)} = {\left({\oa}^{\left({b}^{\left({a}^{{b}^{\prime}}\right)}\right)}\right)}^{\left(-{b}^{a}\right)} = {\vphantom{\left(A^a\right)} \oa}^{\left(-{b}^{a}\right)} = \oa.
	\end{equation*}
	
	Here we applied that if ${\vphantom{\left(A^a\right)} \oa}^{{a}_{1}}=\oa$, then it follows that ${\vphantom{\left(A^a\right)} \oa}^{\left(-{a}_{1}\right)}=\oa$.
	
	We have the following equalities
	\begin{equation*}
		{\vphantom{\left(A^a\right)}\oa}^{\left({a}^{\left(b+b'\right)}-{a}^{\left(b'+b\right)}\right)} = {\left({\vphantom{\left(A^a\right)} \oa}^{a}\right)}^{-a} = {\vphantom{\left(A^a\right)} \oa}^{0} = a,
	\end{equation*}
	where we applied condition 9 of Definition \ref{def:pentcond} and the fact that $-\left({a}^{{b}_{_1}}\right)={(-a)}^{{b}_{1}}$, since 
	\begin{equation*}
		0={0}^{{b}_{1}}={(-a+a)}^{{b}_{1}}={(-a)}^{{b}_{1}}+{a}^{{b}_{1}},
	\end{equation*}
	for any ${(\text{-})}^{{b}_{1}}\in{\bb}_{1}$, ${\bb}_{1}\in\PA(A)$.
\end{proof}

Note that $\St(A)$ is a subobject of $A$, while $\wSt(A)$ is a subset of $A$ in general. 

\begin{proposition}\label{prop:sumofpent}
	Let $A\in\rGwa$ and $\bb,\bb'\in\PA(A)$. If $A$ is perfect in $\rGwa$, then $\bb+\bb'\in\PA(A)$.
\end{proposition}
\begin{proof}
	We shall show that $\bb+\bb'$ defined by \eqref{eq:bpb} satisfies conditions 1--12 of Definition \ref{def:pentcond}.
	\begin{enumerate}[leftmargin=1cm]
		\item[1.] Checking of condition 1: for the left side we have
		\begin{alignat*}{2}
			\left(b+b'\right)\cdot(a+a')&=b\cdot(b'\cdot(a+a'))\\ 
			&=b\cdot(b'\cdot a+b'\cdot a') \\
			&=b\cdot(b'\cdot a)+ b\cdot(b'\cdot a') \\
			&= \left(b+b'\right)\cdot a + \left(b+b'\right)\cdot a',
		\end{alignat*}
		which proves the condition.
		\item[2.] Condition 2: for the left side we have
		\begin{alignat*}{2}
			{(a+a')}^{\left(b+b'\right)}&= {\left({(a+a')}^{b}\right)}^{b'} \\
			&= {\left({a}^{b}\right)}^{b'}+{\left({a'}^{b}\right)}^{b'} \\
			&= {a}^{\left(b+b'\right)}+{a'}^{\left(b+b'\right)},
		\end{alignat*} 
		which ends the proof.
		\item[3.] For $a\neq 0$ we have
		\begin{equation*}
			{\left(\left(b+b'\right)\cdot a\right)}^{a'} = {\left(b\cdot(b'\cdot a)\right)}^{a'} = {\left(b'\cdot a\right)}^{a'} = {a}^{a'},
		\end{equation*}
		here we applied that condition 3 holds for $b$ and $b'$, and therefore proved that this condition for any $A$, not taking into account that
		$A$ is perfect. 
		\item[4.] We have to show 
		\begin{equation*}
			{\left(b+b'\right)}^{(a+a')}={\left({\left(b+b'\right)}^{a}\right)}^{a'}+{\left(b+b'\right)}^{a'}.
		\end{equation*}
		For the left side we have
		\begin{alignat*}{2}
			{\left(b+b'\right)}^{(a+a')} &= {b}^{a+a'}+b\cdot {b'}^{a+a'} \\
			&= {\left({b}^{a}\right)}^{a'}+{b}^{a'}+b\cdot\left({\left({b'}^{a}\right)}^{a'}+{b'}^{a'}\right) \\
			&= {\left({b}^{a}\right)}^{a'}+{b}^{a'}+b\cdot{\left({b'}^{a}\right)}^{a'}+b\cdot\left({b'}^{a'}\right).
		\end{alignat*}
		For the right side we have
		\begin{alignat*}{2}
			{\left({\left(b+b'\right)}^{a}\right)}^{a'}+{\left(b+b'\right)}^{a'} &= {\left({b}^{a}+b\cdot {b'}^{a}\right)}^{a'}+{b}^{a'}+b\cdot \left({b'}^{a'}\right) \\
			&= {\left({b}^{a}\right)}^{a'}+{\left(b\cdot{b'}^{a}\right)}^{a'}+{b}^{a'}+b\cdot\left({b'}^{a'}\right).
		\end{alignat*}
		Since $A$ is perfect ${\left(b\cdot{b'}^{a}\right)}^{a'}+{b}^{a'}={b}^{a'}+{\left(b\cdot{b'}^{a}\right)}^{a'}$ which ends the proof of equality 4. 
		\item[5.] In order to prove this condition we need to show that the following equality is satisfied.
		\begin{equation*}
			{\left({a}^{\left(b+b'\right)\cdot a'}\right)}^{\left(b+b'\right)}={\left({a}^{\left(b+b'\right)}\right)}^{a'}.
		\end{equation*}
		We compute first the left side, which is equal to the following
		\begin{equation*}
			{\left({a}^{b\cdot(b'\cdot a')}\right)}^{\left(b+b'\right)} = {\left({\left({a}^{b\cdot(b'\cdot a')}\right)}^{b}\right)}^{b'} = {\left({\left({a}^{b}\right)}^{b'\cdot a'}\right)}^{b'}.
		\end{equation*}
		Here we applied condition 5 for ${(\text{-})}^{b}$ and ${(\text{-})}^{b'}$. Note that one could apply also the fact that $A$ is perfect, and replace $b'\cdot a'$ by $a'$ by Lemma \ref{lem:perfect}, but we will show that this condition holds for $\bb+\bb'$ for any $A\in\rGwa$. The right side is equal to
		\begin{equation*}
			{\left({\left({a}^{b}\right)}^{b'}\right)}^{a'} = {\left({\left({a}^{b}\right)}^{b'\cdot a'}\right)}^{b'},
		\end{equation*}
		here we applied condition 5 for ${(\text{-})}^{b'}$, which proves the desired equality.
		\item[6.] We have to show that $\left(b+b'\right)\cdot{a}^{a'}={a}^{a'}$ for $a'\neq0$. The left side is equal to $b\cdot\left(b'\cdot {a}^{a'}\right)=b\cdot {a}^{a'}={a}^{a'}$, here we applied condition 6 for $b'\cdot \text{-}$ and $b\cdot \text{-}$.
		\item[7.] Now we will prove the equality
		\begin{equation*}
			{\left(b+b'\right)}^{\left({a}^{a'}\right)}={\left(b+b'\right)}^{a}.
		\end{equation*}
		We have
		\begin{alignat*}{2}
			{\left(b+b'\right)}^{\left({a}^{a'}\right)} &= {b}^{\left({a}^{a'}\right)}+b\cdot {b'}^{\left({a}^{a'}\right)} \\
			&= {b}^{a}+b\cdot {(b')}^{a} \\
			&={\left(b+b'\right)}^{a},
		\end{alignat*}
		since ${b}^{(\text{-})}$ and ${b'}^{(\text{-})}$ satisfy condition 7, which proves the equality.
		\item[8.] The equality 
		\begin{equation*}
			{a}^{\left(b+b'\right)}+a'=a'+{a}^{\left(b+b'\right)}.
		\end{equation*}
		always holds since $A$ is perfect.
		\item[9.] We need to show that
		\begin{equation*}
			{a}^{\left({a'}^{\left(b+b'\right)}\right)}={a}^{a'}.
		\end{equation*}
		The left side is equal to 
		\begin{equation*}
			{a}^{\left(\left({a'}^{b}\right)^{b'}\right)} = {a}^{\left({a'}^{b}\right)} = {a}^{a'}.
		\end{equation*}
		Here we applied condition 9 for ${(\text{-})}^{b}$ and ${(\text{-})}^{b'}$.
		\item[10.] Now we check the condition
		\begin{equation*}
			{a}^{\left({\left(b+b'\right)}^{a'}\right)}=a.
		\end{equation*}
		We compute the left side, which is equal to 
		\begin{equation*}
			{a}^{\left({b}^{a'}+b\cdot{b'}^{a'}\right)}={\left({a}^{\left({b}^{a'}\right)}\right)}^{\left(b\cdot{b'}^{a'}\right)},
		\end{equation*}
		here we will apply that $A$ is perfect and the fact that ${a}^{\left({b}^{a'}\right)}=a$ by condition 10 for $\bb$, which gives that the obtained expression is equal to ${a}^{\left({b'}^{a'}\right)}=a$, since $\bb'$ satisfies condition 10.
		\item[11.] We have to show 
		\begin{equation*}
			(\left(b+b'\right)\cdot a)\cdot \left(b+b'\right)=\left(b+b'\right)\cdot (a\cdot \left(b+b'\right))=a.
		\end{equation*}
		We apply that this condition holds for $\bb$ and $\bb'$ and obtain that the left side is equal to 
		\begin{equation*}
			\left(\left(b\cdot\left(b'\cdot a\right)\right)\cdot b\right)\cdot b' = \left(b'\cdot a\right)\cdot b'=a.
		\end{equation*}
		Note that this equality is proved for any $A$. If we take into account that $A$ is perfect, then it is trivial that the equality holds, since $\bb$ and $\bb'$ satisfy conditions 6 and 6$^{circ}$ of Definition \ref{def:pentcond}.
		\item[12.] Last of all we prove the equality
		\begin{equation*}
			{\vphantom{\left({}^{\left(b+b'\right)}{a}\right)}}^{\left(b+b'\right)}{\left({a}^{\left(b+b'\right)}\right)}={\left({}^{\left(b+b'\right)}{a}\right)}^{\left(b+b'\right)}=a.
		\end{equation*}
		The left side is equal to
		\begin{equation*}
			{\vphantom{\left({\vphantom{\left({\left({a}^{b}\right)}^{b'}\right)}}^{b'}{\left({\left({a}^{b}\right)}^{b'}\right)}\right)}}^{b}{\left({\vphantom{\left({\left({a}^{b}\right)}^{b'}\right)}}^{b'}{\left({\left({a}^{b}\right)}^{b'}\right)}\right)}={\vphantom{\left({a}^{b}\right)}}^{b}{\left({a}^{b}\right)}=a,
		\end{equation*}
		where we apply condition 12 for $\bb$ and $\bb'$.
	\end{enumerate}
	The equalities $1^{\circ},2^{\circ},3^{\circ},5^{\circ},6^{\circ},8^{\circ}$ and $9^{\circ}$ are proved in the similar ways as the equalities $1,2,3,5,6,8$ and $9$, respectively.
\end{proof}

\begin{proposition}
	For any $A\in\rGwa$ and $\bb,\bb'\in\PA(A)$, if $A$ is perfect, then ${\bb}^{\bb'}\in\PA(A)$.
\end{proposition}
\begin{proof}
	We will check the conditions of Definition \ref{def:pentcond} like we did it in the proof of Proposition \ref{prop:sumofpent}.
	\begin{enumerate}[leftmargin=1cm]
		\item[1.] We have ${b}^{b'}\cdot(a+a')={b}^{b'}\cdot a+{b}^{b'}\cdot a'$, since by the definition of ${\bb}^{\bb'}$, ${b}^{b'}\cdot a=a$ for any $a\in A$.
		\item[2.] We have to show
		\begin{equation*}
			{(a+a')}^{\left({b}^{b'}\right)}={a}^{\left({b}^{b'}\right)}+{a'}^{\left({b}^{b'}\right)}.
		\end{equation*}
		The left side is equal to ${(a+a')}^{b}$ by the definition of ${\bb}^{\bb'}$ and we have ${(a+a')}^{b}={a}^{b}+{a'}^{b}$, since $\bb\in\PA(A)$. The right side is equal to the same expression again by the definition of ${\bb}^{\bb'}$.
		\item[3.] For any $a'\neq 0$ we have to show 
		\begin{equation*}
			{({b}^{b'}\cdot a)}^{a'}={a}^{a'}.
		\end{equation*}
		The left side is equal to ${a}^{a'}$ by the definition of ${\bb}^{\bb'}$.
		\item[4.] By the definition of ${\bb}^{\bb'}$ the left side is equal to 
		\begin{alignat*}{2}
			{\left({b}^{b'}\right)}^{(a+a')}&= {\left({b}^{b'\cdot(a+a')}\right)}^{b'} \\
			&= {\left({b}^{b'\cdot a+b'\cdot a'}\right)}^{b'} \\
			&= {\left({\left({b}^{b'\cdot a}\right)}^{b'\cdot a'}+{b}^{b'\cdot a'}\right)}^{b'} \\
			&= {\left({\left({b}^{b'\cdot a}\right)}^{b'}\right)}^{a'} + {\left({b}^{b'\cdot a'}\right)}^{b'},
		\end{alignat*}
		here we applied condition 4 of Definiton \ref{def:pentcond}. By the definition of ${\bb}^{\bb'}$ the right side is equal to ${\left({\left({b}^{b'\cdot a}\right)}^{b'}\right)}^{a'} + {\left({b}^{b'\cdot a'}\right)}^{b'}$, which proves the equality.
		\item[5.] In order to prove this condition we need to show that the following equality is satisfied.
		\begin{equation*}
			{\left({a}^{\left({b}^{b'}\right)\cdot a'}\right)}^{\left({b}^{b'}\right)}={\left({a}^{\left({b}^{b'}\right)}\right)}^{a'}.
		\end{equation*}
		The left side is equal to 
		\begin{equation*}
			{\left({a}^{a'}\right)}^{\left({b}^{b'}\right)}={\left({a}^{a'}\right)}^{b}.
		\end{equation*}
		The right side is equal to
		\begin{equation*}
			{\left({a}^{b}\right)}^{a'}={\left({a}^{b\cdot a'}\right)}^{b}={\left({a}^{a'}\right)}^{b};
		\end{equation*}
		here we applied that $A$ is perfect and therefore $b\cdot a'=a'$, which proves the equality.
		\item[6.] We have ${b}^{b'}\cdot{a}^{a'}={a}^{a'}$ for $a'\neq 0$, by the definition of ${\bb}^{\bb'}$.
		\item[7.] Now we will prove the equality
		\begin{equation*}
			{\left({b}^{b'}\right)}^{\left({a}^{a'}\right)}={\left({b}^{b'}\right)}^{a}.
		\end{equation*}
		The left side is equal to 
		\begin{equation*}
			{\left({b}^{b'\cdot \left({a}^{a'}\right)}\right)}^{b'}={\left({b}^{\left({a}^{a'}\right)}\right)}^{b'}={\left({b}^{a}\right)}^{b'},
		\end{equation*}
		here we applied conditions 6 and 7 of Definition \ref{def:pentcond}. The right side is equal to
		\begin{equation*}
			{\left({b}^{b'\cdot a}\right)}^{b'}={\left({b}^{a}\right)}^{b'},
		\end{equation*}
		since $A$ is perfect and $b'\cdot a=a$, which proves the equality.
		\item[8.] We have to show that if ${a}^{\left({b}^{b'}\right)}\neq a$ for at least one $a\in A$, then
		\begin{equation*}
			{a}^{\left({b}^{b'}\right)}+a'=a'+{a}^{\left({b}^{b'}\right)}.
		\end{equation*}
		The condition ${a}^{\left({b}^{b'}\right)}\neq a$ for at least one $a\in A$ is equivalent to the condition that ${a}^{b}\neq a$ for at least one $a\in A$.
		The left side is equal to 
		\begin{equation*}
			{a}^{\left({b}^{b'}\right)}+a'={a}^{b}+a',
		\end{equation*}
		by the definition of ${\bb}^{\bb'}$ and the condition 8 of Definition \ref{def:pentcond} for $\bb\in\PA(A)$. Note that the equality 8 is true in the case when $A$ is not perfect.
		\item[9.] We have to show that
		\begin{equation*}
			{a}^{\left({a'}^{\left({b}^{b'}\right)}\right)}={a}^{a'}.
		\end{equation*}
		For the left side, by the definition of ${\bb}^{\bb'}$, we have 
		\begin{equation*}
			{a}^{\left({a'}^{\left({b}^{b'}\right)}\right)}={a}^{\left({a'}^{b}\right)}={a}^{a'},
		\end{equation*}
		since $\bb$ is a pentaction.
		\item[10.] Now we check the condition
		\begin{equation*}
			{a}^{\left({\left({b}^{b'}\right)}^{a'}\right)}=a.
		\end{equation*}
		We have
		\begin{equation*}
			{a}^{\left({\left({b}^{b'}\right)}^{a'}\right)}={a}^{\left({\left({b}^{\left(b'\cdot a'\right)}\right)}^{b'}\right)}={a}^{\left({b}^{\left(b'\cdot a'\right)}\right)}=a.
		\end{equation*}
		Here we applied the definition of ${\bb}^{\bb'}$, conditions 9 and 10 for $\bb$ and $\bb'$, respectively, of Definition \ref{def:pentcond}.
	\end{enumerate}
	Condition 11 for ${\bb}^{\bb'}$ follows directly from the definition of ${\bb}^{\bb'}$. Condition 12 follows as well from the definition of ${\bb}^{\bb'}$ and the fact that $\bb$ is a pentaction with property 12.
	The conditions $1^{\circ},2^{\circ},3^{\circ},5^{\circ},6^{\circ},8^{\circ}$ and $9^{\circ}$ are proved in the similar ways as their duals above.
\end{proof}

\begin{theorem}\label{theo:pentactrgwa}
	If $A$ is a perfect object of $\rGwa$ with $\wSt(A)=0$, then $\PA(A)$ is an object of $\rGwa$.
\end{theorem}

\begin{proof}
	First we have to show, that the addition defined in $\PA(A)$ is associative, the function $\bbz$ defined by us is the zero element in this set and there exists an opposite element for any $\bb\in\PA(A)$. After this we will show that the action conditions are satisfied in $\PA(A)$ and that we have two additional conditions as for reduced groups with actions presented in Section \ref{sec:prelim}.
	
	It is easy to see that the addition is associative in $\PA(A)$, i.e. $\left(\bb+\bb'\right)+\bb''=\bb+\left(\bb'+\bb''\right)$, for any $\bb,\bb',\bb''\in\PA(A)$. It is obvious that the zero element is the function $\bbz$ defined by us. For any $\bb\in\PA(A)$ the opposite element is defined by
	\begin{equation*}
		-\bb=\left(\left(-b\right)\cdot \text{- } , \text{-} \cdot \left(-b\right) , {(\text{-})}^{\left(-b\right)},{\vphantom{(\text{-})}}^{\left(-b\right)}{(\text{-})},{\left(-b\right)}^{(\text{-})}\right),
	\end{equation*}
	where
	\begin{alignat*}{2}
		\left(-b\right)\cdot \text{-} &= \text{-} \cdot b \\
		\text{-} \cdot \left(-b\right) &= b\cdot \text{-} \\
		{(\text{-})}^{\left(-b\right)} &= {\vphantom{(\text{-})}}^{b}{(\text{-})} \\
		{\vphantom{(\text{-})}}^{\left(-b\right)}{(\text{-})} &= {(\text{-})}^{b} \\
		{\left(-b\right)}^{(\text{-})} &= -\left[\left({b}^{(\text{-})}\right)\cdot b\right].
	\end{alignat*}
	It is easy to check that $\bb+\left(-\bb\right)=-\bb+\bb=\bbz$.
	
	Now we have to show that the addition defined by us in $\PA(A)$ satisfies action conditions in the category of groups.
	
	Therefore we have to show
	\begin{enumerate}[label={(\alph*)}]
		\item ${\left(\bb+\bb'\right)}^{\bb''}={\bb}^{\bb''}+{\bb'}^{\bb''}$
		\item ${\bb}^{\left(\bb'+\bb''\right)}={\left({\bb}^{\bb'}\right)}^{\bb''}$
		\item ${\bb}^{\bbz}=\bb$
	\end{enumerate}
	for any $\bb,\bb',\bb''\in\PA(A)$.
	\begin{enumerate}[label={(\alph*)}]
		\item We have ${\left(b+b'\right)}^{b''}\cdot a=a=\left({b}^{b''}+{b'}^{b''}\right)\cdot a$ for any ${\left(b+b'\right)}^{b''}\cdot \text{-}\in{\left(\bb+\bb'\right)}^{\bb''}$, any $\left({b}^{b''}+{b'}^{b''}\right)\cdot \text{-}\in{\bb}^{\bb''}+{\bb'}^{\bb''}$ and any $a\in A$.
		We have to show
		\begin{equation*}
			{a}^{\left({\left(b+b'\right)}^{b''}\right)}={a}^{\left({b}^{b''}+{b'}^{b''}\right)}.
		\end{equation*}
		The left side is equal to ${\left({a}^{b}\right)}^{b'}={a}^{\left({b}^{b''}+{b'}^{b''}\right)}$, which proves the equality. Moreover we have to show
		\begin{equation*}
			{\left({\left(b+b'\right)}^{b''}\right)}^{a}={\left({b}^{b''}+{b'}^{b''}\right)}^{a}.
		\end{equation*}
		The left side is equal to
		\begin{alignat*}{2}
			{\left({\left(b+b'\right)}^{(b''\cdot a)}\right)}^{b''} &= {\left({b}^{a}+b\cdot {b'}^{a}\right)}^{b''} \\
			&= {\left({b}^{a}\right)}^{b''}+{\left(b\cdot {b'}^{a}\right)}^{b''} \\
			&= {\left({b}^{a}\right)}^{b''}+{\left({b'}^{a}\right)}^{b''},
		\end{alignat*}
		here we applied Lemma \ref{lem:perfect} since $A$ is perfect.
		
		The right side is equal to
		\begin{alignat*}{2}
			{\left({b}^{b''}\right)}^{a}+{b}^{b''}\cdot {\left({b'}^{b''}\right)}^{a} &= {\left({b}^{b''}\right)}^{a}+{\left({b'}^{b''}\right)}^{a} \\
			&= {\left({b}^{a}\right)}^{b''}+{\left({b'}^{a}\right)}^{b''},
		\end{alignat*}
		here we applied Corollary \ref{cor:changeofpower} which proves the desired equality.
		
		The equalities for the dual actions are proved similarly.
		\item We have to show 
		\begin{equation*}
			{b}^{\left(b'+b''\right)}\cdot a = {\left({b}^{b'}\right)}^{b''}\cdot a,
		\end{equation*}
		for any ${b}^{\left(b'+b''\right)}\cdot \text{-}\in{\bb}^{\bb'+\bb''}$ and ${\left({b}^{b'}\right)}^{b''}\cdot \text{-}\in{\left({\bb}^{\bb'}\right)}^{\bb''}$. Both sides are equal to $a$ by the definition of action operation in $\PA(A)$. The equality
		\begin{equation*}
			{a}^{\left({b}^{\left(b'+b''\right)}\right)}={a}^{\left({\left({b}^{b'}\right)}^{b''}\right)}
		\end{equation*}
		holds by the definition of action operation ib $\PA(A)$.
		
		Now we have to show
		\begin{equation*}
			{\left({b}^{\left(b'+b''\right)}\right)}^{a}={\left({\left({b}^{b'}\right)}^{b''}\right)}^{a}.
		\end{equation*}
		Both sides are equal to ${\left({\left({b}^{\left(b'\cdot\left(b''\cdot a\right)\right)}\right)}^{b'}\right)}^{b''}$ without application of the condition that $A$ is perfect, which gives the desired equality.
		\item We have to show 
		\begin{equation*}
			{b}^{0}\cdot a = b\cdot a.
		\end{equation*}
		The left side is equal to $a$ by the definition of action operation in $\PA(A)$, and the right side is equal to $a$ as well by Lemma \ref{lem:perfect} since $A$ is perfect.
		\begin{equation*}
			{a}^{\left({b}^{0}\right)}={a}^{b}
		\end{equation*}
		by the definition of action ${\bb}^{\bbz}$.
		
		We have to show that
		\begin{equation*}
			{\left({b}^{0}\right)}^{a}={b}^{a}.
		\end{equation*} 
		The left side is equal to
		\begin{equation*}
			{\left({b}^{0\cdot a}\right)}^{a}={\left({b}^{a}\right)}^{0}={b}^{a},
		\end{equation*}
		which gives the equality. For the dual actions the corresponding equalities are proved similarly.
	\end{enumerate}
	Now we have to show that in $\PA(A)$ we have 
	\begin{enumerate}[label={(\roman{*})}]
		\item ${\bb}^{\bb'}+\bb''=\bb''+{\bb}^{\bb'}$ for $\bb'\neq\bbz$
		\item ${\bb}^{\left({\bb'}^{\bb''}\right)}={\bb}^{\bb'}$
	\end{enumerate}
	for any $\bb,\bb'$ and $\bb''\in\PA(A)$.
	\begin{enumerate}[label={(\roman{*})}]
		\item We have to show
		\begin{equation*}
			\left({b}^{b'}+b''\right)\cdot a = \left(b''+{b}^{b'}\right)\cdot a,
		\end{equation*}
		for any $\left({b}^{b'}+b''\right)\cdot \text{-}\in{\bb}^{\bb'}+\bb''$, any $\left(b''+{b}^{b'}\right)\cdot \text{-}\in\bb''+{\bb}^{\bb'}$ such that $\bb'\neq\bbz$, and any $a\in A$.
		
		The both sides are equal to $b''\cdot a$ by the definition, and the equality holds for any $A$, without being perfect.
		
		We have to show 
		\begin{equation*}
			{a}^{\left({b}^{b'}+b''\right)}={a}^{\left(b''+{b}^{b'}\right)}.
		\end{equation*}
		The left side is equal to 
		\begin{equation*}
			{\left({a}^{\left({b}^{b'}\right)}\right)}^{b''}={\left({a}^{b}\right)}^{b''}={a}^{b+b''}.
		\end{equation*}
		And the right side is equal to ${a}^{b''+b}$. Here we apply that $\wSt(A)=0$, which implies that these two elements are equal.
		
		Moreover we have to show
		\begin{equation*}
			{\left({b}^{b'}+b''\right)}^{a}={\left(b''+{b}^{b'}\right)}^{a}.
		\end{equation*}
		The left side is equal to
		\begin{equation*}
			{\left({b}^{b'}\right)}^{a}+{b}^{b'}\cdot \left({b''}^{a}\right)={\left({b}^{b'}\right)}^{a}+{b''}^{a},
		\end{equation*}
		since $A$ is perfect and the right side is equal to 
		\begin{equation*}
			{b''}^{a}+b''\cdot{\left({b}^{b'}\right)}^{a}={b''}^{a}+{\left({b}^{b'}\right)}^{a}.
		\end{equation*}
		Since $A$ is perfect these two elements commute and the corresponding sums are equal. For the dual actions the equalities are checked similarly. 
		
		Now we have to check equality (ii).
		\item For the dot action it is obvious 
		\begin{equation*}
			{b}^{\left({b'}^{b''}\right)}\cdot a={b}^{b'}\cdot a=a,
		\end{equation*}
		for any $a\in A$ and any elements from ${\bb}^{\left({\bb'}^{\bb''}\right)}$ and ${\bb}^{\bb'}$.
		
		We have to show 
		\begin{equation*}
			{a}^{\left({b}^{\left({b'}^{b''}\right)}\right)}={a}^{\left({b}^{b'}\right)}.
		\end{equation*}
		The both sides are equal to ${a}^{b}$ by the definition of action operation in $\PA(A)$. 
		
		It is left to show 
		\begin{equation*}
			{\left({b}^{\left({b'}^{b''}\right)}\right)}^{a}={\left({b}^{b'}\right)}^{a},
		\end{equation*}
		for any elements of ${\bb}^{\left({\bb'}^{\bb''}\right)}$ and ${\bb}^{\bb'}$, and for any $a\in A$. The left side is equal to
		\begin{equation*}
			{\left({b}^{\left(\left({b'}^{b''}\right)\cdot a\right)}\right)}^{\left({b'}^{b''}\right)} = {\left({b}^{a}\right)}^{\left({b'}^{b''}\right)} = {\left({b}^{a}\right)}^{b'}.
		\end{equation*}
		The right side is equal to 
		\begin{equation*}
			{\left({b}^{\left(b'\cdot a\right)}\right)}^{b'}={\left({b}^{a}\right)}^{b'},
		\end{equation*}
		by Corollary \ref{cor:changeofpower} since $A$ is perfect. For the dual actions the corresponding equalities are proved similarly, which ends the proof of Theorem \ref{theo:pentactrgwa}.
	\end{enumerate}
\end{proof}

\begin{example}
	Let $A$ be an abelian group with trivial action on itself, then $A\in\rGwa$. Denote by $A_0$ the abelian subgroup of $A$ generated by $\wSt(A)$ and denote the quotient group $A/A_0$ by $A_1$. Suppose $A_0\neq A$ and the natural epimorphism $A\rightarrow A_1$ has a section. Then there is an isomorphism $A\cong A_0\oplus A_1$. This is the case e.g. when $A$ is a vector space over a field and $A_0$ its subspace. For simplicity we identify $A$ with the direct sum $A_0\oplus A_1$. Any pentaction $\bb\in\PA(A_1)$ defines functions
	\begin{equation*}
		\tilde{b}\cdot \text{-} ~,~  {(\text{-})}^{\tilde{b}} ~,~  {\tilde{b}}^{(\text{-})} \colon A\rightarrow A.
	\end{equation*}
	in the following way
	\begin{alignat*}{2}
		\tilde{b}\cdot (a_0,a_1) &= (a_0,b\cdot a_1)\\
		{(a_0,a_1)}^{\tilde{b}} &= \left(a_0,{a_1}^{b}\right) \\
		{\tilde{b}}^{(a_0,a_1)} &= \left(0,{b}^{a_1}\right)
	\end{alignat*}
	for any $(a_0,a_1)\in A_0\oplus A_1$. Analogously for the dual actions. Easy checking shows that $\widetilde{\bb}=\left(\tilde{b}\cdot \text{-},\text{-}\cdot \tilde{b},{(\text{-})}^{\tilde{b}},{\vphantom{(-)}}^{\tilde{b}}{(\text{-})},{\tilde{b}}^{(\text{-})}\right)$ is a pentaction. Any element of $\wSt(A_1)$ defines an element of $\wSt(A_0\oplus A_1)$; e.g. for ${a_1}^{b+b'}-{a_1}^{b'+b}$ we have 
	\begin{equation*}
		{(0,a_1)}^{\tilde{b}+\tilde{b}'}-{(0,a_1)}^{\tilde{b}'+\tilde{b}}=\left(0,{a_1}^{b+b'}-{a_1}^{b'+b}\right),
	\end{equation*}
	from which we obtain that ${a_1}^{b+b'}-{a_1}^{b'+b}=0$ since $A_0=\left\langle \wSt(A) \right\rangle$.Other cases are considered similarly. Here it is worth to note that every object of $\rGwa$ with trivial action on itself is perfect.
\end{example}

\begin{remark}
The idea of constructing objects in $\rGwa$ with zero weak stabilizer leads us to the definition and study of Noetherian objects in this category. Note that Noetherian groups are already defined and studied nowadays.
\end{remark}

\section{Action of $\PA(A)$ on $A$ and action representability in $\rGwa$}\label{sec:actrep}

Let $A\in\rGwa$. An action of $\PA(A)$ on $A$ is defined in a natural way. For any $\bb=\left(b\cdot \text{- },\text{-}\cdot b,{(\text{-})}^{b},{\vphantom{(\text{-})}}^{b}{(\text{-})},{b}^{(\text{-})}\right)$ we define 
\begin{equation*}
	\bb\cdot a=b\cdot a,
\end{equation*}
which is the value of the function $b\cdot \text{-}\colon A\rightarrow A$ on the element $a\in A$. Analogously 
\begin{alignat*}{2}
	{a}^{\bb}&={a}^{b}\\
	{\bb}^{a}&={b}^{a}.
\end{alignat*}
By the definition of $-\bb$, we will have
\begin{alignat*}{2}
	(-\bb)\cdot a &= a\cdot b \\
	{a}^{-\bb} &= {\vphantom{a}}^{b}{a} \\
	{\left(-\bb\right)}^{a} &=-\left({b}^{a}\cdot b\right).
\end{alignat*}

\begin{theorem}\label{theo:pentactderact}
	The action of $\PA(A)$ on $A$ is a derived action in $\rGwa$ if $A$ is perfect and $\wSt(A)=0$, $A\in\rGwa$.
\end{theorem}
\begin{proof}
	We have to check that group action conditions are satisfied, also conditions $\boldsymbol{(1_{A})}-\boldsymbol{(4_{A})}$, $\boldsymbol{(1_{B})}-\boldsymbol{(4_{B})}$, ${a}_{1}-{a}_{10}$ and ${a}^{\bbz}=a$, stated in Section \ref{sec:prelim}, where $a\in A$.
	
	Checking of group action conditions, 
	\begin{enumerate}[label={$\bullet$}]
		\item $\left(\bb+\bb'\right)\cdot a=\bb\cdot\left(\bb'\cdot a\right)$ is true by the definition of $\bb+\bb'$ and the definition of the action of pentactions on $A$.
		\item $\bb\cdot\left(a+a'\right)=\bb\cdot a+\bb\cdot a'$, since it is a property of a pentaction.
		\item $\bbz\cdot a=a$ is true by the definition of the pentaction $\bbz$ and the definition of action of pentactions on $A$.
	\end{enumerate}
	
	Checking of conditions $\boldsymbol{(1_{A})}-\boldsymbol{(4_{A})}$,
	\begin{enumerate}
		\item[$\boldsymbol{(1_{A})}$] This condition is true by the definition of pentaction, condition 2. 
		\item[$\boldsymbol{(2_{A})}$] This condition is true by the definition of the sum $\bb+\bb'$ of pentactions.
		\item[$\boldsymbol{(3_{A})}$] We have ${\left(\bb\cdot a\right)}^{a'}={a}^{a'}$ for any $a,a'\in A$, with $a'\neq 0$ by the definition of pentaction, condition 3.
		\item[$\boldsymbol{(4_{A})}$] ${\left(\bb\cdot a\right)}^{\bb'}={a}^{\bb'}$ is true by the definition of action of $\PA(A)$ on $A$ and Lemma \ref{lem:perfect}, since $A$ is perfect. 
	\end{enumerate}
	
	Checking of conditions $\boldsymbol{(1_{B})}-\boldsymbol{(4_{B})}$,
	\begin{enumerate}
		\item[$\boldsymbol{(1_{B})}$] This condition is true by the definition of pentaction, condition 4.
		\item[$\boldsymbol{(2_{B})}$] This condition is true by the definition of the sum $\bb+\bb'$.
		\item[$\boldsymbol{(3_{B})}$] This condition is true by the definition of pentaction, property 5.
		\item[$\boldsymbol{(4_{B})}$] This condition is true by the definition of ${\bb}^{\bb'}$.
	\end{enumerate}
	Obviously ${a}^{\bbz}=a$, for any $a\in A$.
	
	Checking of the conditions ${a}_{1}-{a}_{10}$. 
	
	\begin{enumerate}[label={${a}_{\arabic{*}}$.}]
		\item Condition ${a}_{1}$ is true by Definition \ref{def:pentcond}, property 6.
		\item This condition is true since $A$ is perfect.
		\item This condition is true by the definition of ${\bb}^{\bb'}$.
		\item This condition is true by Definition \ref{def:pentcond}, property 7.
		\item This condition is true by the definition of ${\bb}^{\bb'}$.
		\item This condition is true by Definition \ref{def:pentcond}, condition 8.
		\item This condition is true by Definition \ref{def:pentcond}, condition 9.
		\item This condition is true by Definition \ref{def:pentcond}, condition 10.
		\item ${\bb}^{\left({\bb'}^{a}\right)}=0$ is true since $\wSt(A)=0$.
		\item ${\bb}^{\left({a}^{\bb'}\right)}={\bb}^{a}$ is true since $\wSt(A)=0$.
	\end{enumerate}
	Therefore all conditions of derived actions in $\rGwa$ are satisfied, which proves the theorem.
\end{proof}

\begin{theorem}\label{theo:pentactrepresent}
	Let $A\in\rGwa$. If $A$ is perfect and $\wSt(A)=0$, then $A$ is action representable and \begin{equation*}
		\AR(A)=\PA(A).
	\end{equation*}
\end{theorem}

\begin{proof}
	Let $B\in\rGwa$ which has a derived action on $A$. Define a function 
	\begin{equation*}
		\varphi\colon B\rightarrow \PA(A)
	\end{equation*}
	in the following way. For any $b\in B$
	\begin{equation*}
		\varphi(b)=\left(b\cdot \text{- }, \text{-}\cdot b, {(\text{-})}^{b}, {\vphantom{(\text{-})}}^{b}{(\text{-})}, {b}^{(\text{-})}\right),
	\end{equation*}
	where $b\cdot a$ is an action of $b$ on $a$; $a\cdot b$ is an action of $-b$ on $a$, i.e. $a\cdot b=(-b)\cdot a$; ${a}^{b}$ is an action of $b$ on $a$; ${\vphantom{a}}^{b}{a}={a}^{-b}$, where on the right side is the action of $-b$ on $a$; ${b}^{a}$ is an action of $b$ on $a$, for any $a\in A$. Easy checking shows that $\varphi(b)$ is a pentaction of $A$; $\varphi$ is a homomorphism of $\rGwa$, it follows from the definition of the operations in $\PA(A)$ and the construction of $\varphi$. We have
	\begin{equation}\label{eq:uniquefunc}
		\begin{alignedat}{2}
			b\cdot a &= \varphi(b)\cdot a \\
			{a}^{b} &= {a}^{\varphi(b)} \\
			{b}^{a} &= {\varphi(b)}^{a}.
		\end{alignedat}
	\end{equation}
	$\varphi$ is unique function with properties \eqref{eq:uniquefunc}, which follows from the definition of pentactions and conditions \eqref{eq:uniquefunc}.
\end{proof}

\section*{Acknowledgements}

It was the second author's idea to define and describe actions in the category of groups with action. It was stimulating and then interesting to search such a subcategory in this category, where it would be possible to investigate action representability of certain objects.

%%\bibliography{Sahan}
%%\bibliographystyle{gmj}

\end{document}